\documentclass[12pt]{article}
\usepackage[utf8]{inputenc}
\usepackage[a4paper]{geometry}
\usepackage{amssymb}
\usepackage{mathrsfs}
\usepackage{stmaryrd}
\usepackage{amsthm}
\usepackage{latexsym}
\usepackage{amsmath}
\usepackage[french]{babel}
\usepackage{amsfonts}
\usepackage{wasysym}
\usepackage{epsfig}
\usepackage{graphicx}
\usepackage{enumitem,nicefrac}
\usepackage{tocloft}
\usepackage{fancyhdr}
\usepackage{dsfont}
\usepackage[all]{xy}
\usepackage{diagbox}
\usepackage{array}
\usepackage[T1]{fontenc}
\usepackage{tgpagella}
\usepackage{hyperref}
\usepackage{float}

\date{}

\newcolumntype{R}[1]{>{\raggedleft\arraybackslash}m{#1}} 
\newcolumntype{L}[1]{>{\raggedright\arraybackslash}m{#1}} 
\newcolumntype{C}[1]{>{\centering\arraybackslash}m{#1}}

\newcommand{\Rev}[0]{\textup{Rev}}

\newcommand{\Id}[0]{\textup{Id}}
\newcommand{\Iso}[0]{\textup{Iso}}

\newcommand{\Coul}[0]{\textup{Coul}}
\newcommand{\Bla}[0]{\textup{Bla}}
\newcommand{\Ble}[0]{\textup{Ble}}
\renewcommand{\O}[0]{\textup{O}}
\newcommand{\R}[0]{\textup{R}}
\newcommand{\V}[0]{\textup{V}}
\newcommand{\J}[0]{\textup{J}}

\newcommand{\Faces}[0]{\textup{Faces}}
\newcommand{\Aut}[0]{\textup{Aut}}

\newcommand{\Ineq}[0]{\textup{Ineq}}
\newcommand{\fib}[1]{#1\textup{fib}}
\newcommand{\tor}[1]{#1\textup{tor}}
\newcommand{\ArGeom}[0]{\textup{ArGeom}}
\newcommand{\SomGeom}[0]{\textup{SomGeom}}
\newcommand{\CentGeom}[0]{\textup{CentGeom}}
\newcommand{\ArMat}[0]{\textup{ArMat}}
\newcommand{\SomMat}[0]{\textup{SomMat}}
\newcommand{\CentMat}[0]{\textup{CentMat}}

\newcommand{\Ad}[0]{\textup{Ad}}

\theoremstyle{plain}
    \newtheorem{Theo}{Théorème}[section]
    \newtheorem{Cor}[Theo]{Corollaire}
    \newtheorem{Prop}[Theo]{Proposition}
    \newtheorem{Lem}[Theo]{Lemme}
\theoremstyle{definition}
    \newtheorem{Def}[Theo]{Définition}
    \newtheorem*{Ex}{Ex}
    
\theoremstyle{remark}
    \newtheorem*{Rem}{Remarque}
    \newtheorem*{Rems}{Remarques}
\theoremstyle{remark}
    \newtheorem*{Not}{Notations}

\title{Le calcul du nombre d'états inéquivalents du Rubik's Revenge}
\author{Victor Le Guilloux}

\begin{document}

\maketitle

\begin{abstract}
    Après avoir traduit le problème de résolution du Rubik's Revenge en termes d'actions de groupes, on utilise le résultat de structure des transformations licites de \cite{La} pour dénombrer les états du Rubik's Revenge à transformations licites et indiscernables près, en présence et en l'absence de contraintes mécaniques. On retrouve également par une nouvelle méthode le calcul réalisé dans \cite{BonzioRev} de la probabilité de pouvoir résoudre le Rubik's Revenge en l'ayant monté au hasard, encore une fois en présence et en l'absence de contraintes mécaniques.
\end{abstract}

\section*{Introduction}

Le Rubik's Revenge est un casse-tête similaire au célèbre Rubik's Cube : il consiste en un cube dont les faces sont divisées en $4\times4$ cellules carrées que l'on doit (presque, en raison de l'existence de pièces dont les coloriages les rend indiscernables) remettre dans leur position initiale après les avoir mélangées. Pour les mélanger et les remettre dans leur état initial on ne considère comme licite que l'utilisation des rotations des tranches du casse-tête. Ces transformations engendrent un groupe que nous noterons $\mathcal{L}$. Ce groupe a été calculé dans \cite{Colmez} pour le Rubik's Cube, dans \cite{La} pour le Rubik's Revenge, et pour des cubes de taille arbitraire dans \cite{Bonzio}.

Le Rubik's Revenge possède en outre une structure mécanique interne suffisamment complexe pour ne pas permettre de le monter dans toutes les configurations imaginables. La prise en compte ou non de cette mécanique conduit naturellement à des résultats différents lorsqu'on étudie le casse-tête.\\

On peut alors légitimement se poser les questions suivantes : \textbf{(A)} combien y a-t-il de configurations du cube, à transformations licites et des pièces indiscernables près si on ne prend pas en compte les contraintes imposées par la mécanique du Rubik's Revenge ? Et \textbf{(B)} en démontant et remontant le casse-tête au hasard sans s'intéresser aux contraintes mécaniques, quelles sont les chances de l'avoir remonté dans une configuration où on peut le résoudre ? \textbf{(A')} Qu'en est-il de la qestion \textbf{(A)} si on prend en compte ces contraintes ? \textbf{(B')} De même, de quelle façon le résultat de la question \textbf{(B)} est-il modifié si on s'intéresse aux contraintes ?\\

Les questions \textbf{(B)} et \textbf{(B')} ont été résolues dans \cite{BonzioRev} mais on va utiliser ici une approche différente pour les résoudre.\\

Dans le cas du Rubik's Cube, la première observation à faire est qu'il n'y a pas de contraintes mécaniques particulières en raison de la confection des pièces. Les questions \textbf{(A)} et \textbf{(A')} sont donc équivalentes, de même que les questions \textbf{(B)} et \textbf{(B')}. De plus, comme il n'existe pas de pièces indiscernables, les deux questions \textbf{(A)} et \textbf{(B)} sont en réalité équivalentes et on peut les traiter comme suit. Si on dispose d'un couple d'objets isomorphes d'une catégorie $\mathcal{C}$, on peut lui associer un \emph{torseur} (l'ensemble des isomorphismes entre ces deux objets) ; dans la suite $\mathcal{C}$ sera la catégorie des \emph{fibrés principaux discrets} (introduits dans la section \ref{I}). On introduit des objets $\ArGeom,\ArMat,\SomGeom,\SomMat$ de $\mathcal{C}$ (désignant respectivement les arêtes géométriques et matérielles et les sommets géométriques et matériels), avec $\ArGeom\cong \ArMat$ et $\SomGeom\cong \SomMat$, et les torseurs associés sont donc les ensembles $\Iso(\ArGeom,\ArMat)$ et $\Iso(\SomGeom,\SomMat)$. On introduit aussi un groupe $\mathcal{L}\subset\Aut(\ArGeom)\times\Aut(\SomGeom)$, et le cardinal recherché dans \textbf{(A)} est alors celui de :
\begin{align*}
    (\Iso(\ArGeom,\ArMat)\times\Iso(\SomGeom,\SomMat))/\mathcal{L},
\end{align*}
qui est également celui de $(\Aut(\ArGeom)\times\Aut(\SomGeom))/\mathcal{L}$. Ce cardinal est calculé dans \cite{Colmez} (Théorème 1). D'autre part on répond à la question \textbf{(B)} en montrant que la probabilité recherchée est donnée par l'inverse de ce cardinal (l'utilisation de l'expression "au hasard" sous-entend que l'on considère la loi uniforme sur $\Iso(\ArGeom,\ArMat)\times\Iso(\SomGeom,\SomMat)$).\\

La situation est plus compliquée dans le cas du Rubik's Revenge. En plus de devoir remplacer le produit  $\Iso(\ArGeom,\ArMat)\times\Iso(\SomGeom,\SomMat)$ par le produit :
\begin{align*}
    \Iso(\ArGeom,\ArMat)\times\Iso(\SomGeom,\SomMat)\times\Iso(\CentGeom,\CentMat)
\end{align*}
(où $\CentGeom,\CentMat$ sont des objets de $\mathcal{C}$ isomorphes représentant respectivement les centres géométriques et les centres matériels, pour tenir compte des pièces centrales), certaines des pièces ont un coloriage qui les rend indiscernables (les 24 arêtes sont regroupées en 12 paires d'arêtes indiscernables et les 24 centres sont regroupés en 6 ensembles de 4 centres indiscernables), et la confection des pièces rend certaines configurations impossibles à réaliser sans casser le Rubik's Revenge. Pour prendre en compte les pièces indiscernables, on les numérote, mais on doit maintenant prendre en considération les transformations qui peuvent permuter ces pièces indiscernables. Pour répondre aux questions \textbf{(A)} et \textbf{(B)} on est alors amené à modifier le cadre de travail en remplaçant les torseurs par des \emph{bitorseurs} en ajoutant une action à gauche de $\Aut(\ArMat)\times\Aut(\SomMat)\times\Aut(\CentMat)$. Pour pouvoir résoudre les questions \textbf{(A')} et \textbf{(B')} on doit en plus se restreindre à l'action à droite et à gauche de sous-groupes des produits d'automorphismes des structures au départ et à l'arrivée. 

Dans les section \ref{II} à \ref{IV}, nous introduisons les objets dont il est question (pièces géométriques et matérielles) et l'ensemble des coloriages $\Rev_{marq}$ (section \ref{II}), les sous-groupes $\mathcal{I}\subset\mathcal{T}_{mat}$ correspondant aux transformations des pièces indiscernables (section \ref{III}) et $\mathcal{L}\subset\mathcal{T}_{geom}$ des transformations licites (section \ref{IV}). Dans la section \ref{V}, nous reformulons en ces termes et nous résolvons les questions \textbf{(A)} et \textbf{(B)} ; plus précisément on résout \textbf{(A)} en calculant le cardinal du double quotient $\mathcal{I}\backslash\Rev_{marq}/\mathcal{L}$, puis si $x_0$ désigne l'état initial du Rubik's Revenge, on répond à \textbf{(B)} en calculant $\frac{\#\mathcal{I}\cdot x_0\cdot\mathcal{L}}{\#\Rev_{marq}}$, ce qui donne la probabilité recherchée. Enfin, dans la section \ref{VI} nous reformulons les contraintes mécaniques en termes d'actions de groupes et nous répondons à la question \textbf{(A')} en calculant le cardinal d'un autre double quotient, $\mathcal{I}'\backslash\Rev_{meca}/\mathcal{L}$, où $\mathcal{I}'$ est un sous-groupe de $\mathcal{I}$ et $\Rev_{meca}$ est un sous-ensemble de $\Rev_{marq}$ contenant tous deux les informations des contraintes mécaniques, et on répond à la question \textbf{(B')} en calculant $\frac{\#\mathcal{I'}\cdot x_0\cdot\mathcal{L}}{\#\Rev_{meca}}$.

\section{Préliminaires algébriques}\label{I}

Si $\mathcal{C}$ est une catégorie et $A,B$ sont deux objets de $\mathcal{C}$ isomorphes, alors les groupes $\Aut(A)$ et $\Aut(B)$ agissent respectivement à droite et à gauche sur l'ensemble $\Iso(A,B)$ de façon libre et transitive, et les actions commutent. En faisant agir un seul de ces groupes on obtient un \emph{torseur}, et en faisant agir les deux groupes on obtient un \emph{bitorseur}. Dans cette section nous allons détailler les propriétés des torseurs et nous construisons une catégorie, celle des \emph{fibrés principaux discrets}, une structure qui sera adaptée à l'étude du Rubik's Revenge.

\begin{Not}
    Etant donné un ensemble $E$ doté d'une action à gauche d'un groupe $G$ et d'une action à droite d'un groupe $G'$, on notera $G\backslash E = \{G\cdot e|e\in E\}$. De façon analogue on notera $E/G'$ les orbites de l'action à droite.
    
    Si $B$ est un ensemble, on notera $G\wr\mathfrak{S}(B)$ le produit en couronne de $G$ par $\mathfrak{S}(B)$.
\end{Not}

\subsection{Torseurs et bitorseurs}

\begin{Def}
    Soient $G$ un groupe et $E$ un ensemble non vide. On dit que $E$ est un \emph{$G$-torseur} à gauche (resp. à droite) si $G$ agit librement et transitivement à gauche (resp. à droite) sur $E$.
    
    Si $G'$ est un groupe, on dit que $E$ est un \emph{$(G,G')$-bitorseur} si $E$ est un $G$-torseur à gauche et un $G'$-torseur à droite et si les actions de $G$ et $G'$ commutent.
\end{Def}

\begin{Def}
    Soient $G$ un groupe et $E,F$ des $G$-torseurs à gauche. Un \emph{morphisme de $G$-torseurs} est une application $\varphi : E \longrightarrow F$ telle que pour tous $g\in G, e \in E$, $\varphi(g\cdot e) = g\cdot\varphi(e)$.
    
    Si $\varphi$ est une bijection on dit que c'est un \emph{isomorphisme de $G$-torseurs}. On note $\Iso_{\tor{G-}}(E,F)$ l'ensemble des isomorphismes de $G$-torseurs $E\longrightarrow F$.
    
    Un isomorphisme de $G$-torseurs $E\longrightarrow E$ est appelé un \emph{automorphisme de $G$-torseur}. On note $\Aut_{\tor{G-}}(E)$ le groupe des automorphismes de $G$-torseur de $E$.
    
    On définit de même la notion de morphisme de torseurs pour des $G$-torseurs à droite.
\end{Def}

\begin{Def}
    Soit $G$ un groupe. On définit une loi de groupe $*$ sur $G$ en posant pour tous $g,h\in G$, $g*h=hg$. Le groupe $(G,*)$ est appelé \emph{groupe opposé de $G$} et est noté $G^{op}$.
\end{Def}

\begin{Rem}
    L'application $g\longmapsto g^{-1}$ est un isomorphisme de groupes $G\longrightarrow G^{op}$.
\end{Rem}

\begin{Prop}\label{EquivCatTorBitor}
    \begin{enumerate}[label=\textup{(\roman*)}]
        \item\label{EquivCatTorBitor1} Soit $E$ un $G$-torseur. Alors $E$ est un $(G,\Aut_{\tor{G-}}(E)^{op})$-bitorseur.
        \item\label{EquivCatTorBitor2} Si $E$ est un $(G,G')$-bitorseur, alors $G'\cong\Aut_{\tor{G-}}(E)^{op}$.
    \end{enumerate}
\end{Prop}

\begin{proof}
    \ref{EquivCatTorBitor1} Voir \cite{Giraud}, chapitre III paragraphe 1.5.3.1.\\
    \ref{EquivCatTorBitor2} Voir \cite{Giraud}, chapitre III proposition 1.5.4.
\end{proof}

\begin{Prop}\label{isom}
    \begin{enumerate}[label=\textup{(\roman*)}]
        \item\label{isom1} Si $E$ est un $G$-torseur, alors le choix d'un élément $e\in E$ induit un isomorphisme $\Aut_{\tor{G-}}(E)\cong G^{op}$.
        \item\label{isom2} Si $E$ est un $(G,G')$-bitorseur, alors le choix d'un élément $e\in E$ induit des isomorphismes inverses l'un de l'autre $\Ad_e : G\longrightarrow G'$ et $\Ad'_e : G'\longrightarrow G$ tels que pour tous $g\in G,g'\in G'$, $e\cdot\Ad_e(g)=g\cdot e$ et $\Ad'_e(g')\cdot e = e\cdot g'$.
    \end{enumerate}
\end{Prop}

\begin{proof}
    \ref{isom1} On vérifie immédiatement que poser, pour $g\in G$, $\Phi_e(g)$ l'unique automorphisme de torseurs vérifiant $\Phi_e(g)(e) = g\cdot e$ définit un isomorphisme de groupes $G^{op}\longrightarrow\Aut_{\tor{G-}}(E)$.\\
    \ref{isom2} Découle de \ref{isom1} et de \ref{EquivCatTorBitor}.\ref{EquivCatTorBitor2}.
\end{proof}

\begin{Prop}\label{SousBitor}
    Soient $E$ un $(G,G')$-bitorseur, $H'\subset G'$ un sous-groupe et $e \in E$. Alors $e\cdot H'$ est un $(\Ad'_e(H'),H')$-bitorseur.
\end{Prop}

\subsection{Isomorphismes de bitorseurs}

\begin{Def}
    Soient $E$ un $(G,G')$-bitorseur et $F$ un $(H,H')$-bitorseur. On dit que $E$ et $F$ sont \emph{isomorphes} s'il existe des isomorphismes de groupes $\varphi : G\longrightarrow H,\varphi' : G'\longrightarrow H'$ et une bijection $\psi : E\longrightarrow F$ tels que pour tous $x\in E,g\in G,g'\in G'$ on ait $\psi(g\cdot x) = \varphi(g)\cdot\psi(x)$ et $\psi(x\cdot g') = \psi(x)\cdot \varphi'(g')$.
    
    On dira alors que $(\varphi,\psi,\varphi')$ est un \emph{isomorphisme de bitorseurs entre $E$ et $F$}.
\end{Def}

\begin{Rem}
    Si $G=H$ (resp. si $G'=H'$) et si $\varphi=\Id_G$ (resp. si $\varphi'=\Id_{G'}$) alors $\psi$ est un isomorphisme de $G$-torseurs à gauche (resp. un isomorphisme de $G'$-torseurs à droite).
\end{Rem}

\begin{Prop}\label{IsoBitorseur}
    Soit $E$ un $(G,G')$-bitorseur. Le choix d'un élément $e\in E$ fournit :
    \begin{enumerate}[label=\textup{(\roman*)}]
        \item  un isomorphisme de bitorseurs $(\Id_G,\psi_e,\Ad'_e)$ entre $E$ et le $(G,G)$-bitorseur trivial $G$ vérifiant pour tous $x \in E,g'\in G'$, $\psi_e(x)\cdot e = x$ ;
        \item un isomorphisme de bitorseurs $(\Ad_e,\psi'_e,\Id_{G'})$ entre $E$ et le $(G',G')$-bitorseur trivial $G'$ vérifiant pour tous $x \in E,g\in G$, $e\cdot\psi'_e(x) = x$.
    \end{enumerate}
\end{Prop}

\begin{proof}
    Immédiat.
\end{proof}

\begin{Rem}
    Si on dispose d'un isomorphisme de groupes $G\longrightarrow H$, alors on a un isomorphisme de bitorseurs entre $E$ et le $(H,H)$-bitorseur trivial $H$.
\end{Rem}

\subsection{Ensembles cycliques}

Soit $n\in\mathbf{N}^*$. Dans toute la suite on notera $C_n = \mathbf{Z}/n\mathbf{Z}$.

\begin{Def}
    Un \emph{ensemble cyclique} est un couple $(E,\sigma)$ où $E$ est un ensemble fini et où $\sigma\in \mathfrak{S}(E)$ est un $\#E$-cycle.
    
    Si $E=\{e_1,\ldots,e_n\}$ et $\sigma = (e_1,\ldots,e_n)$ on note $(E,\sigma) = [e_1,\ldots,e_n]$.
\end{Def}

\begin{Ex}
    Tout ensemble à deux éléments possède une unique structure d'ensemble cyclique.
\end{Ex}

\begin{Prop}
    Soient $E$ un ensemble fini. La donnée d'une structure cyclique sur $E$ est équivalente à la donnée d'une structure de $C_{\#E}$-torseur.
\end{Prop}

\begin{proof}
    Immédiat.
\end{proof}

\subsection{Fibrés principaux discrets}

Soit $G$ un groupe.

\begin{Def}
    Un \emph{fibré discret} est la donnée de deux ensembles $E$ et $B$ appelés respectivement l'\emph{espace total} et la \emph{base} et d'une application surjective $p : E\longrightarrow B$. Dans toute la suite on parlera de fibrés à la place de fibrés discrets.
    
    Si $x\in B$, on appelle \emph{fibre au dessus de $x$} l'ensemble $E_x=p^{-1}(\{x\})$.
    Un \emph{morphisme de fibrés} entre deux fibrés $(E,B,p)$ et $(E',B',p')$ est un couple $(\varphi,\Phi)$ où $\varphi : B\longrightarrow B'$ et $\Phi : E\longrightarrow E'$ sont deux applications telles que $\varphi\circ p = p'\circ\Phi$.
    
    Dans la suite, s'il n'y a pas d'ambiguïté, on se référera aux fibrés seulement par leur base ou par leur espace total.
\end{Def}

\begin{Rem}
    On définit ainsi la catégorie des fibrés (discrets). Un morphisme $(\varphi,\Phi)$ est alors un isomorphisme si et seulement si $\varphi$ et $\Phi$ sont des bijections.
\end{Rem}

\begin{Def}
    On notera $\Iso_{\fib{}}((E,B,p),(E',B',p'))$ l'ensemble des isomorphismes de fibrés $(E,B,p)\longrightarrow(E',B',p')$ (resp. $\Aut_{\fib{}}(E,B,p)$ l'ensemble des automorphismes de fibré de $(E,B,p)$).
\end{Def}

\begin{Lem}
    Soit $(\varphi,\Phi) : (E,B,p)\longrightarrow(E',B',p')$ un morphisme de fibrés. Alors pour tout $x\in B,\Phi(E_x) \subset E'_{\varphi(x)}$.
\end{Lem}

\begin{proof}
    Évidente.
\end{proof}

\begin{Def}
    Un \emph{$G$-fibré} $(E,B,p)$ est un fibré où les fibres de $E$ sont des $G$-ensembles. Si les fibres sont des $G$-torseurs, on dit que $(E,B,p)$ est un \emph{$G$-fibré principal}.
    
    Un \emph{morphisme de $G$-fibrés} est un morphisme de fibrés $(\varphi,\Phi)$ tel que la restriction de $\Phi$ à chaque fibre est un $G$-morphisme.
\end{Def}

\begin{Rems}
    $\centerdot$ La catégorie des $G$-fibrés principaux (discrets) est équivalente à la catégorie des $G$-ensembles libres.\\
    $\centerdot$ Si $(E,B,p)$ est un $G$-fibré transitif (i.e. si l'action de $G$ sur chaque fibre de $E$ est transitive), alors $B$ s'identifie à $G\backslash E$.
\end{Rems}

\begin{Def}
    Soit $(E,B,p)$ un fibré. Un \emph{marquage de $(E,B,p)$}, ou une \emph{section de $(E,B,p)$} est une application $s : B \longrightarrow E$ telle que $p\circ s = \Id_B$
\end{Def}

\begin{Prop}
    Soit $(E,B,p)$ un $G$-fibré principal. La donnée d'un marquage $s$ de $(E,B,p)$ donne un isomorphisme de groupes $\varphi_s :\Aut_{\fib{G-}}(E)\longrightarrow\mathfrak{S}(B)\ltimes G^B$ tel qu'en notant $\varphi_s(\sigma,\Phi) =(\sigma,(g_x)_{x\in B})$ ($(\sigma,\Phi)\in\Aut_{\fib{G-}}(E)$), pour tout $x\in B$, $g_x$ soit l'unique élément de $G$ vérifiant $\Phi(g_x\cdot s(x)) = s(\sigma(x))$.
\end{Prop}

\begin{proof}
    Immédiat.
\end{proof}

\begin{Cor}\label{Structure}
    Soit $(E,B,p)$ un $G$-fibré principal. La donnée d'un marquage $s$ de $(E,B,p)$ donne un isomorphisme de groupes $\phi_s : \Aut_{\fib{G-}}(E) \longrightarrow G\wr\mathfrak{S}(B)$ tel qu'en notant $\phi_s(\sigma,\Phi) =((g'_x)_{x\in B},\sigma)$ ($(\sigma,\Phi)\in\Aut_{\fib{G-}}(E)$), pour tout $x\in B$, $g'_x$ soit l'unique élément de $G$ vérifiant $g'_{\sigma(x)}\cdot \Phi(s(x)) = s(\sigma(x))$.
\end{Cor}

\begin{proof}
    Immédiat.
\end{proof}

\section{Pièces géométriques et pièces matérielles du Rubik's Revenge}\label{II}

Dans cette section nous allons mettre en évidence la structure de bitorseur du Rubik's Revenge. Plus précisément nous allons définir séparément les pièces géométriques (arêtes, sommets et centres agencés dans l'espace mais non coloriés) et les pièces matérielles (arêtes, sommets et centres coloriés mais non agencés dans l'espace). Ces pièces (géométriques et matérielles) vont former des produits de fibrés principaux dont nous nous servirons dans la suite.

\subsection{Pièces géométriques}

\begin{Def}
    Le Rubik's Revenge non colorié est appelé le \emph{Rubik's Revenge géométrique}. Les pièces non coloriées du Rubik's Revenge sont appelées les \emph{pièces géométriques}. On note $\ArGeom$ (resp. $\SomGeom$, resp. $\CentGeom$) l'ensemble des \emph{pièces géométriques d'arête} (resp. les \emph{pièces géométriques de sommet}, resp. \emph{les pièces géométriques de centre}).
    
    On note $\Faces$ l'ensemble des faces des pièces géométriques du Rubik's Revenge et on note $p$ l'application de $\Faces$ dans l'ensemble des pièces géométriques qui à une face $F$ associe la pièce géométrique qu'elle borde. On notera encore $p$ les restrictions de $p$ à $\ArGeom$, $\SomGeom$ et $\CentGeom$.
    
    On notera $\Faces_\ArGeom =p^{-1}(\ArGeom),\Faces_\SomGeom= p^{-1}(\SomGeom)$ et $\Faces_{\CentGeom}=p^{-1}(\CentGeom)$.
\end{Def}

\begin{Lem}
    $\Faces$ est partitionné en 56 sous-ensembles suivant la pièce que les faces bordent.
\end{Lem}

\begin{proof}
    Évidente.
\end{proof}

\begin{Def}
    Si $\alpha$ est une pièce géométrique du Rubik's Revenge, on note $\Faces_\alpha$ l'ensemble des faces qui bordent la pièce $\alpha$.
\end{Def}

\begin{Rem}
    Si $\alpha\in\ArGeom$ (resp. si $\alpha\in\SomGeom$, resp. $\alpha\in \CentGeom$), $\#\Faces_\alpha=2$ (resp. $\#\Faces_\alpha=3$, resp. $\#\Faces_\alpha=1$).
\end{Rem}

\begin{Lem}
    Les triplets $(\Faces_\ArGeom,\ArGeom,p),(\Faces_\SomGeom,\SomGeom,p)$ et $(\Faces_\CentGeom,\CentGeom,p)$ sont des fibrés. Plus précisément $\Faces_\ArGeom$ est un $C_2$-fibré principal, $\Faces_\CentGeom$ est un fibré principal trivial et la structure de $C_3$-fibré principal de $\SomGeom$ est donnée comme suit. Soit $\beta\in\SomGeom$. On définit un axe orienté de $\beta$ vers le centre du Rubik's Revenge, et alors la rotation d'un tiers de tour (dans le sens direct) autour de cet axe induit un 3-cycle $\sigma_\beta\in\mathfrak{S}(\Faces_\beta)$ (voir la figure \ref{FigRotSomGeom}).
    
    On identifiera le fibré $(\Faces_\CentGeom,\CentGeom,p)$ à l'ensemble $\CentGeom$.
\end{Lem}

\begin{proof}
    Comme les fibres de $\ArGeom$ sont toutes de cardinal 2, il n'existe qu'une seule structure de $C_2$-torseur sur chaque fibre.
\end{proof}

\begin{figure}[H]
    \centering
    \includegraphics[width = 10cm]{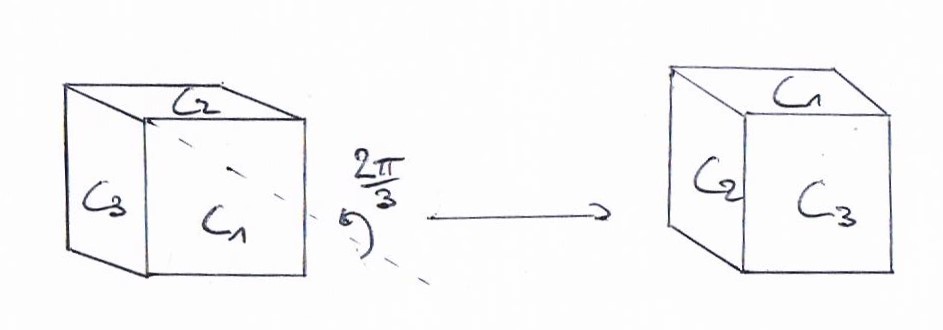}
    \caption{3-cycle induit par la rotation du sommet géométrique}\label{FigRotSomGeom}
\end{figure}

\subsection{Pièces matérielles}

\begin{Def}
    Le Rubik's Revenge colorié, où on oublie les informations sur les positions des pièces entre elles, est appelé le \emph{Rubik's Revenge matériel}. Quitte à numéroter les pièces coloriées pour les différencier, on note $\ArMat$ (resp. $\SomMat$, resp. $\CentMat$) l'ensemble des \emph{pièces matérielles d'arête} (resp. les \emph{pièces matérielles de sommet}, resp. \emph{les pièces matérielles de centre}).
\end{Def}

\begin{Rem}
    Les mêmes résultats s'appliquent exactement de la même façon aux pièces matérielles qu'aux pièces géométriques. De même que pour les sommets géométriques, pour doter $\SomMat$ d'une structure fibrée principale sur $C_3$, il suffit de choisir une structure cyclique sur chaque fibre.
\end{Rem}

\begin{Def}
    On choisira les structures cycliques suivantes sur les sommets matériels : $[\Bla,\Ble,\R]$, $[\Bla,\O,\Ble]$, $[\Bla,\V,\O]$, $[\Bla,\R,\V]$, $[\J,\V,\O]$, $[\J,\Ble,\O]$,\\ $[\J,\V,\R]$ et  $[\J,\O,\V]$.
\end{Def}

\begin{Rem}
    \`A la différence des pièces géométriques, certaines pièces matérielles sont indiscernables.
\end{Rem}

\begin{Def}
    On note $\Coul$ l'ensemble $\Coul=\{\Bla,\J,\Ble,\V,\O,\R\}$.
    
    On note $\pi_\mathcal{A} : \ArMat \longrightarrow \mathcal{P}_2(\Coul)$ l'application qui à une arête matérielle associe l'ensemble de ses couleurs. De même on note $\pi_\mathcal{C} : \CentMat\longrightarrow\Coul$ l'application qui à un centre matériel associe sa couleur.
    
    Deux pièces matérielles de même(s) couleur(s) (i.e. telles que $\pi_\mathcal{A}(\alpha)=\pi_\mathcal{A}(\alpha')$ pour deux arêtes matérielles $\alpha$ et $\alpha'$, ou $\pi_\mathcal{C}(\gamma)=\pi_\mathcal{C}(\gamma')$ pour deux centres matériels $\gamma$ et $\gamma'$) sont dites \emph{indiscernables}.
\end{Def}

\begin{Rem}
    Il n'existe pas de sommets matériels indiscernables.
\end{Rem}

\subsection{Les coloriages du Rubik's Revenge comme produit d'isomorphismes de fibrés principaux entre pièces géométriques et pièces matérielles}

Nous allons maintenant définir ce qu'est un \emph{coloriage du Rubik's Revenge} en nous servant des structures de fibrés que nous avons construites précédemment.

\begin{Def}
    Un \emph{coloriage marqué du Rubik's Revenge} est un triplet $(\mathfrak{f},\mathfrak{g},h)$ où $\mathfrak{f}:\ArGeom\longrightarrow\ArMat, \mathfrak{g}:\SomGeom\longrightarrow\SomMat$ et $h:\CentGeom\longrightarrow\CentMat$ sont des isomorphismes de fibrés principaux. 
    
    On note $\Rev_{marq}$ l'ensemble des coloriages marqués du Rubik's Revenge. On notera dans la suite $\mathcal{T}\ArGeom = \Aut_{\fib{C_2-}}(\ArGeom), \mathcal{T}\ArMat = \Aut_{\fib{C_2-}}(\ArMat)$, \\$\mathcal{T}\SomGeom = \Aut_{\fib{C_3-}}(\SomGeom)$, et $\mathcal{T}\SomMat = \Aut_{\fib{C_3-}}(\SomMat)$.
    
    Enfin, on appelle \emph{groupe des transformations géométriques} le groupe $\mathcal{T}_{geom} = \mathcal{T}\ArGeom\times\mathcal{T}\SomGeom\times\mathfrak{S}(\CentGeom)$, et on appelle \emph{groupe des transformations matérielles} le groupe $\mathcal{T}_{mat} = \mathcal{T}\ArMat\times\mathcal{T}\SomMat\times\mathfrak{S}(\CentMat)$.
\end{Def}

\begin{Rems}
    $\centerdot$ $\CentGeom$ et $\CentMat$ étant des fibrés principaux sur le groupe trivial, un isomorphisme de fibrés principaux entre eux n'est rien d'autre qu'une bijection.\\
    $\centerdot$ Les coloriages sont dits \emph{marqués} car on a numéroté les pièces matérielles pour pouvoir les distinguer.
\end{Rems}

\section{Le groupe des automorphismes des pièces matérielles et le sous-groupe des transformations des pièces indiscernables}\label{III}

La différence notable entre le Rubik's Cube et le Rubik's Revenge est que, contrairement au Cube, le Rubik's Revenge possède des pièces indiscernables. Un groupe agit naturellement sur ces pièces, en les permutant.

\begin{Def}
    On note $\widetilde{\mathcal{I}_\mathcal{A}}$ le sous-groupe de $\mathfrak{S}(\ArMat)$ des éléments $\sigma$ vérifiant $\pi_\mathcal{A}\circ\sigma=\pi_\mathcal{A}$.
\end{Def}

Dans la suite, si deux arêtes matérielles $\alpha$ et $\alpha'$ (resp. deux centres matériels $\gamma$ et $\gamma'$) sont indiscernables, on identifiera les fibres au dessus de $\alpha$ et $\alpha'$ à deux copies de $\pi_\mathcal{A}(\alpha)=\pi_\mathcal{A}(\alpha')$.

\begin{Lem}
    Soit $\varphi : \widetilde{\mathcal{I}_\mathcal{A}} \longrightarrow \mathcal{T}\ArMat, \sigma \longmapsto (\sigma,\Phi)$, où pour tout $\alpha\in\ArMat$, $\Phi_{|p^{-1}(\{\alpha\})}=\Id_{\pi_\mathcal{A}(\alpha)}$. Alors $\varphi$ est un morphisme de groupes injectif.
\end{Lem}

\begin{proof}
    Par définition si $\sigma\in\widetilde{\mathcal{I}_\mathcal{A}}$, on peut identifier les fibres au dessus de $\alpha$ et de $\sigma(\alpha)$ pour toute arête $\alpha$ et donc le relèvement d'une permutation $\sigma\in\widetilde{\mathcal{I}_\mathcal{A}}$ en un morphisme de la forme $(\sigma,\Phi)$ a bien un sens puisque $\Id_{\pi_\mathcal{A}(\alpha)}$ est un isomorphisme de $C_2$-torseurs $\pi_\mathcal{A}(\alpha)\longrightarrow\pi_\mathcal{A}(\sigma(\alpha))$ (et donc l'application $\varphi$ est bien définie). De plus $\varphi$ est un morphisme de groupes, si $\pi : \mathcal{T}\ArMat\longrightarrow \mathfrak{S}(\ArMat)$ est la projection canonique, alors la composée $\pi\circ\varphi$ est l'injection canonique $\widetilde{\mathcal{I}_\mathcal{A}}\longrightarrow\mathfrak{S}(\ArMat)$, ce qui montre l'injectivité de $\varphi$.
\end{proof}

\begin{Def}
    On appelle \emph{groupe des transformations indiscernables des arêtes} et on note $\mathcal{I}(\ArMat)$ le groupe $\varphi(\widetilde{\mathcal{I}_\mathcal{A}})$.
    
    Le \emph{groupe des transformations indiscernables des centres}, noté $\mathcal{I}(\CentMat)$, est le sous-groupe de $\mathfrak{S}(\CentMat)$ des éléments $\phi$ vérifiant $\pi_\mathcal{C}\circ\phi=\pi_\mathcal{C}$.
    
    On appelle \emph{groupe des transformations indiscernables} le sous-groupe de $\mathcal{T}_{mat}$ $\mathcal{I}=\mathcal{I}(\ArMat)\times\{1\}\times\mathcal{I}(\CentMat)$.
\end{Def}

\begin{Rem}
    On a les isomorphismes $\mathcal{I}(\ArMat) \cong \mathfrak{S}_2^{12}, \mathcal{I}(\CentMat)\cong \mathfrak{S}_4^6,$ et $\mathcal{I}\cong \mathfrak{S}_2^{12}\times \mathfrak{S}_4^6$.
\end{Rem}

\section{Le groupe des automorphismes des pièces géométriques et le sous-groupe des transformations licites}\label{IV}

De même qu'un groupe agit naturellement sur les pièces indiscernables, un autre groupe agit naturellement sur les pièces géométriques, en les transformant en respectant les règles du jeu.

\subsection{Définition des générateurs des transformations licites}

Dans toute la suite nous considérerons le marquage $s$ des arêtes géométriques que l'on peut observer sur la figure \ref{FigMarquageAretes}, et on prendra $\varsigma$ un marquage quelconque des sommets.
\begin{figure}[H]
    \centering
    \includegraphics[width = 5cm]{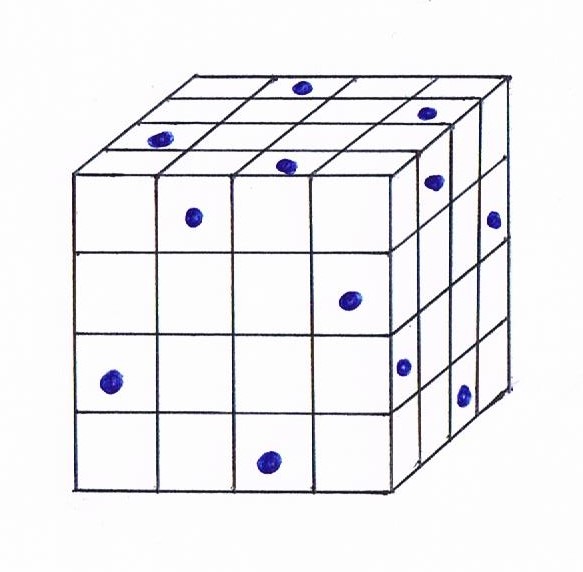}
    \caption{Choix du marquage des arêtes}\label{FigMarquageAretes}
\end{figure}

Les choix de ces marquages nous permettent d'identifier $\mathcal{T}_{geom}$ au groupe $(C_2\wr\mathfrak{S}(\ArGeom))\times (C_3\wr\mathfrak{S}(\SomGeom))\times \mathfrak{S}(\CentGeom)$.

\begin{Def}
    Les \emph{mouvements élémentaires de tranche du Rubik's Revenge} sont notés $B, MB, MF, F, L, ML, MR, R, D, MD, MU, U$.
    
    Ils engendrent un sous-groupe de $\mathcal{T}_{geom}$, noté $\mathcal{L}$ et appelé le \emph{groupe des transformations licites du Rubik's Revenge}.
    
    On note $\mathcal{T}'_{geom} = \big\{\big((\sigma,\Phi),(\tau,\Psi),\phi)\in\mathcal{T}_{geom}\big|\forall \alpha\in\ArGeom,\Phi(s(\alpha))=s(\sigma(\alpha))\big\}$, le sous-groupe de $\mathcal{T}_{geom}$ des transformations qui préservent le marquage $s$.
\end{Def}

\begin{Prop}
    $\mathcal{L}$ est un sous-groupe de $\mathcal{T}'_{geom}$.
\end{Prop}

\begin{proof}
    Les générateurs du groupe $\mathcal{L}$ sont des éléments de $\mathcal{T}'_{geom}$.
\end{proof}

\begin{Def}
    On note $\mathcal{T}_{geom}^\wr= (C_2\wr\mathfrak{S}(\ArGeom))\times(C_3\wr\mathfrak{S}(\SomGeom))\times\mathfrak{S}(\CentGeom)$.
     
    On note $\mathcal{T}^{\prime\wr}_{geom}=(\{0\}\wr\mathfrak{S}(\ArGeom))\times(C_3\wr\mathfrak{S}(\SomGeom))\times\mathfrak{S}(\CentGeom) \subset\mathcal{T}_{geom}^\wr$.
\end{Def}

\begin{Lem}
    L'isomorphisme $\mathcal{T}_{geom}\longrightarrow\mathcal{T}_{geom}^\wr$ induit par les choix des marquages des arêtes géométriques et des sommets géométriques se restreint en un isomorphisme $\mathcal{T}'_{geom}\longrightarrow\mathcal{T}_{geom}^{\prime\wr}$.
\end{Lem}

\begin{proof}
    Évidente.
\end{proof}

\begin{Def}
    On note $\mathcal{L}^\wr$ l'image de $\mathcal{L}$ par l'isomorphisme $\mathcal{T}_{geom}\longrightarrow\mathcal{T}_{geom}^\wr$ induit par les choix des marquages.
\end{Def}

Dans toute la suite on travaillera avec $\mathcal{L}^\wr$ plutôt qu'avec $\mathcal{L}$.

\subsection{Caractérisation des transformations licites}

\begin{Def}
    On appelle \emph{morphisme caractéristique} et on note $\chi$ le morphisme $\chi :\mathcal{T}^{\prime\wr}_{geom}\longrightarrow C_3\times \{-1,1\} , (\sigma,(\theta,\tau),\phi) \longmapsto \big(\sum_\beta\theta_\beta,\varepsilon(\tau)\varepsilon(\phi)\big)$.
\end{Def}

\begin{Theo}[Caractérisation des transformations licites du Rubik's Revenge]
    Le morphisme $\chi$ caractérise $\mathcal{L}^\wr$, dans le sens où $\mathcal{L}^\wr=\ker(\chi)$.
\end{Theo}

\begin{proof}
    Voir \cite{La}.
\end{proof}

\begin{Cor}\label{ind1}
    $\mathcal{L}$ est un sous-groupe de $\mathcal{T}'_{geom}$ d'indice 6.
\end{Cor}

\begin{Cor}
    $\mathcal{L}$ est un sous-groupe de $\mathcal{T}_{geom}$ d'indice $6.2^{24}$.
\end{Cor}

\begin{Cor}\label{ind2}
    La restriction de l'action de $\mathcal{T}_{geom}$ sur $\Rev_{marq}$ à $\mathcal{L}$ donne une action libre à $6.2^{24}$ orbites, toutes de même cardinal.
\end{Cor}

\section{Une première réponse aux problématiques}\label{V}

Dans les sections précédentes nous avons décrit le Rubik’s Revenge en termes
d’actions de groupes. C’est donc en termes d’actions de groupes que nous allons
répondre aux questions \textbf{(A)} et \textbf{(B)}. Ainsi \textbf{(A)} se reformule de la manière suivante
: quel est le cardinal de $\mathcal{I}\backslash\Rev_{marq}/\mathcal{L}$ ?

La question \textbf{(B)} se reformule quant à elle : que vaut $\frac{\#\mathcal{I}\cdot x_0\cdot\mathcal{L}}{\#\Rev_{marq}}$ ?

\subsection{Calcul d'isomorphismes}

Dans la suite nous allons faire le choix d'un élément $x\in \Rev_{marq}$ pour pouvoir utiliser les différents isomorphismes qu'on a construit dans la section \ref{I} de manière à préserver la structure de bitorseur de $\Rev_{marq}$. Nous allons donc choisir l'élément $x=x_0$, état \emph{initial}, ou \emph{résolu}, du Rubik's Revenge.

\begin{Def}
    On notera $\mathcal{I}^\wr\subset\mathcal{T}_{geom}^\wr$ l'image de $\mathcal{I}\subset\mathcal{T}_{mat}$ par l'application $(\phi_s\times\phi'_\varsigma\times\Id_{\mathfrak{S}(\CentGeom)})\circ \Ad_{x_0}$, où $\phi_s$ (resp. $\phi'_\varsigma$) est l'isomorphisme $\mathcal{T}\ArGeom\cong C_2\wr\mathfrak{S}(\ArGeom)$ (resp. l'isomorphisme $\mathcal{T}\SomGeom\cong C_3\wr\mathfrak{S}(\SomGeom)$) obtenu grâce au corollaire \ref{Structure} en prenant le marquage des arêtes géométriques $s$ décrit précédemment (resp. en prenant le marquage des sommets géométriques quelconque $\varsigma$).
\end{Def}

\begin{Rem}
    $x_0$ s'écrivant sous la forme $x_0=(\mathfrak{f},\mathfrak{g},h)$ où $\mathfrak{f}=(f,F)$ et $\mathfrak{g}=(g,G)$ sont des isomorphismes et $h$ est une bijection, $x_0^{-1}$ a bien un sens et alors l'application $\Ad_{x_0}$ est la conjugaison par $x_0^{-1}$.
\end{Rem}

\begin{Lem}\label{CaracI}
    $i=\big(((\rho_\alpha),\sigma),(0,\Id),\phi\big)\in\mathcal{I^\wr}$ si et seulement si les conditions suivantes sont vérifiées :
    \begin{enumerate}[label=\textup{(\roman*)}]
        \item\label{CaracI1} $\sigma\in f^{-1}\widetilde{\mathcal{I}_\mathcal{A}}f$ ;
        \item\label{CaracI2} $\phi\in h^{-1}\mathcal{I}(\CentMat)h$ ;
        \item\label{CaracI3} si $\alpha\neq\alpha'$ sont des arêtes géométriques telles que $f(\alpha)$ et $f(\alpha')$ sont indiscernables, alors $\left\{\begin{array}{c}
                 \sigma_{|\{\alpha,\alpha'\}}=\Id \Rightarrow \rho_\alpha=\rho_{\alpha'}=0 \\
                 \sigma_{|\{\alpha,\alpha'\}}\neq\Id \Rightarrow \rho_\alpha=\rho_{\alpha'}=1
            \end{array}\right.$.
    \end{enumerate}
\end{Lem}

\begin{proof}
    Si $i\in\mathcal{I}^\wr$, \ref{CaracI1} et \ref{CaracI2} sont évidemment vérifiées par définition de $\mathcal{I}$ et $\mathcal{I}^\wr$. Nous allons utiliser l'observation suivante pour montrer que si $i\in\mathcal{I}^\wr$, alors on a \ref{CaracI3} : avec le marquage $s$ des arêtes géométriques que l'on a fait, pour toute paire d'arêtes géométriques $\alpha,\alpha'$ telles que $f(\alpha)$ et $f(\alpha')$ sont indiscernables, on a $F(s(\alpha))\neq F(s(\alpha'))$.
    
    Or d'après le corollaire \ref{Structure}, si $((\rho_\alpha)_{\alpha\in\ArGeom},\sigma)$ est l'image d'un isomorphisme $(\sigma',\Phi)\in\mathcal{I}(\ArMat)$, alors $\sigma = f^{-1}\sigma'f$ et pour tout $\alpha\in\ArGeom$, $\rho_{f^{-1}\sigma'f(\alpha)}\cdot s(\alpha) = s(f^{-1}\sigma'f(\alpha))$, i.e. pour tout $\alpha\in\ArGeom$, $\rho_{\sigma(\alpha)}\cdot s(\alpha) = s(\sigma(\alpha))$,
    et comme $\rho_\alpha=\rho_\alpha^{-1}$ pour tout $\alpha\in\ArGeom$, cela équivaut à ce que pour toute arête géométrique $\alpha\in\ArGeom$, $F(s(\alpha)) = \rho_{\sigma(\alpha)}\cdot F(s(\sigma(\alpha))$.
    
    Donc en utilisant l'observation faite en début de démonstration, si $\sigma_{|\{\alpha,\alpha'\}}=\Id$, alors on doit avoir $\rho_\alpha=\rho_{\alpha'}=0$, et si $\sigma_{|\{\alpha,\alpha'\}}\neq\Id$ on doit avoir $\rho_\alpha=\rho_{\alpha'}=1$.\\
    
    Réciproquement si \ref{CaracI1} et \ref{CaracI2} sont vérifiées, alors on a bien $\pi_\mathcal{A}\circ f\sigma f^{-1} = \pi_\mathcal{A}$ et $\pi_\mathcal{C}\circ h\phi h^{-1}=\pi_\mathcal{C}$. D'après \ref{CaracI1} $f\sigma f^{-1}$ se relève donc en un automorphisme de la forme $(f\sigma f^{-1},\Psi)$ qui s'envoie sur $(\rho,\sigma)$ si \ref{CaracI3} est vérifiée.
\end{proof}

\begin{Def}
    On note $T = (C_2\wr\mathfrak{S}_{24})\times(C_3\wr\mathfrak{S}_8)\times\mathfrak{S}_{24}$. On note également $T' = (\{0\}\wr\mathfrak{S}_{24})\times(C_3\wr\mathfrak{S}_8)\times\mathfrak{S}_{24}\subset T$ et $L$ le sous-groupe des éléments $\big((0,\sigma),(\theta,\tau),\phi\big)\in T'$ tels que $\sum_{i=1}^8\theta_i=0$ et $\varepsilon(\tau)=\varepsilon(\phi)$.
\end{Def}

\begin{Rem}
    Fixer des bijections $\ArGeom\longrightarrow\llbracket1,24\rrbracket$, $\SomGeom\longrightarrow\llbracket1,8\rrbracket$ et $\CentGeom\longrightarrow\llbracket1,24\rrbracket$ donne un isomorphisme $\mathcal{T}^\wr_{geom}\longrightarrow T$ qui se restreint à $\mathcal{T}^{\prime\wr}_{geom}$ et à $\mathcal{L}^\wr$ en des isomorphismes $\mathcal{T}^{\prime\wr}_{geom}\longrightarrow T'$ et $\mathcal{L}^\wr\longrightarrow L$.
\end{Rem}

\begin{Def}
    On note $I\subset T$ le sous-groupe des éléments $((\rho,\sigma),(0,\Id),\phi)\in T$ vérifiant les trois propriétés suivantes :
    \begin{enumerate}[label=\textup{(\roman*)}]
        \item $\forall 1\leq k\leq12, \sigma(\{2k-1,2k\}) =\{2k-1,2k\}$ ;
        \item $\forall 1\leq i \leq 6,\phi(\llbracket4k-3,4k\rrbracket) = \llbracket4k-3,4k\rrbracket$ ;
        \item $\forall 1\leq k\leq12, 
        \left\{\begin{array}{c}
             \sigma_{|\{2k-1,2k\}} = \Id \Rightarrow \rho_{2k-1}=\rho_{2k}=0 \\
             \sigma_{|\{2k-1,2k\}} \neq \Id \Rightarrow \rho_{2k-1}=\rho_{2k}=1
        \end{array}\right.$
    \end{enumerate}
\end{Def}

\begin{Lem}
    On peut choisir des bijections $\ArGeom\longrightarrow\llbracket1,24\rrbracket$ et $\CentGeom\longrightarrow\llbracket1,24\rrbracket$ de façon à ce que l'isomorphisme $\mathcal{T}_{geom}^\wr\longrightarrow T$ se restreigne en un isomorphisme $\mathcal{I}^\wr\longrightarrow I$.
\end{Lem}

\begin{proof}
    Évidente.
\end{proof}

\begin{Def}
    On note $I_2 =\{((\overline{0},\overline{0}),\Id),((\overline{1},\overline{1}),(12))\}\subset C_2\wr\mathfrak{S}_2$.
\end{Def}

\begin{Cor}\label{IsomI}
    On a un isomorphisme $I\cong I_2^{12}\times \mathfrak{S}_4^6$.
\end{Cor}

\subsection{Calcul du double quotient : états inéquivalents du Rubik's Revenge}

\begin{Def}
    On appelle \emph{ensemble des états inéquivalents du Rubik's Revenge} le double quotient $\Ineq = \mathcal{I}\backslash\Rev_{marq}/\mathcal{L}$.
\end{Def}

\begin{Prop}
    On a une bijection $\Ineq\cong I\backslash T/L$.
\end{Prop}

\begin{proof}
    On compose tous les isomorphismes de bitorseurs et de groupes obtenus précédemment.
\end{proof}

\begin{Lem}\label{LemIneq}
    On a $\#I\backslash T/T' = 3^{12}$.
\end{Lem}

\begin{proof}
    $\centerdot$ On identifie $C_2^{24}$ à l'ensemble des fonctions $\llbracket1,24\rrbracket\longrightarrow C_2$. $\mathfrak{S}_{24}$ agit à gauche sur $C_2^{24}$ par $(\sigma,c)\longmapsto \sigma\cdot c = c\circ\sigma^{-1}$ (c'est d'ailleurs cette action qu'on utilise dans le produit en couronne $C_2\wr\mathfrak{S}_{24}$). Alors $C_2\wr\mathfrak{S}_{24}$ agit à gauche sur $C_2^{24}$ par $((\rho,\sigma),c)\longmapsto (\rho,\sigma)\cdot c = \rho+\sigma\cdot c$.
    
    En effet on a clairement $(0,\Id)\cdot c=c$ pour tout $c\in C_2^{24}$, et si $(\rho,\sigma),(\theta,\tau)\in C_2\wr \mathfrak{S}_{24}$ :
    \begin{align*}
        \big((\rho,\sigma)(\theta,\tau)\big)\cdot c & = (\rho+\sigma\cdot\theta,\sigma\tau)\cdot c = \rho+\sigma\cdot\theta+(\sigma\tau)\cdot c = \rho +\sigma\cdot(\theta+\tau\cdot c) \\
        & = (\rho,\sigma)\cdot((\theta,\tau)\cdot c)
    \end{align*}
    On fait donc agir $I$ sur $C_2^{24}$ par $\big(((\rho,\sigma),(0,\Id),\phi),c\big)\longmapsto((\rho,\sigma),(0,\Id),\phi)\cdot c = (\rho,\sigma)\cdot c$.\\
    $\centerdot$ Soit $f : C_2^{24} \longrightarrow T/T', c \longmapsto ((c,\Id),(0,\Id),\Id)T'$. $f$ est une bijection, montrons que c'est un morphisme de $I$-ensembles. Soient $c\in C_2^{24}$ et $i=((\rho,\sigma),(0,\Id),\phi)\in I$. On a :
    \begin{align*}
        i\cdot f(c) & = ((\rho,\sigma),(0,\Id),\phi)((c,\Id),(0,\Id),\Id)T' = ((\rho+\sigma\cdot c,\sigma),(0,\Id),\phi)T'\\
        & = ((\rho+\sigma\cdot c,\Id),(0,\Id),\Id)((0,\sigma),(0,\Id),\phi)T' = ((i\cdot c,\Id),(0,\Id),\Id)T' \\
        & = f(i\cdot c)
    \end{align*}
    $f$ est donc un isomorphisme de $I$-ensembles. On a donc $\#I\backslash T/T' = \#I\backslash C_2^{24}$.\\
    $\centerdot$ Calculons $\#I\backslash C_2^{24}$. L'action de $I$ sur $C_2^{24}$ peut être considérée comme la puissance 12-ème de l'action de l'action de $I_2$ sur $C_2^2$, cette dernière action étant donnée par $((\rho,\sigma),c) \longmapsto (\rho,\sigma)\cdot c = \rho+\sigma\cdot c$.
    
    $I_2$ et $C_2^2$ étant de petits cardinaux, on peut calculer explicitement les orbites :
    \begin{align*}
        \begin{array}{cc}
            ((\overline{1},\overline{1}),(12))\cdot(\overline{0},\overline{0}) = (\overline{1},\overline{1}), & ((\overline{1},\overline{1}),(12))\cdot(\overline{1},\overline{1}) = (\overline{0},\overline{0}), \\ ((\overline{1},\overline{1}),(12))\cdot(\overline{0},\overline{1}) = (\overline{0},\overline{1}), & ((\overline{1},\overline{1}),(12))\cdot(\overline{1},\overline{0}) = (\overline{1},\overline{0}) 
        \end{array}
    \end{align*}
    L'action de $I_2$ sur $C_2^2$ est donc une action à 3 orbites, dont un système de représentants est donné par $(\overline{0},\overline{0}),(\overline{0},\overline{1})\text{ et }(\overline{1},\overline{0})$.
    
    Il s'ensuit que l'action de $I$ sur $C_2^{24}$ est une action à $3^{12}$ orbites, et donc l'action de $I$ sur $T/T'$ possède également $3^{12}$ orbites.
\end{proof}

\begin{Theo}\label{QuA}
    On a $\#I\backslash T/L=3^{13}$.
\end{Theo}

\begin{proof}
    $\centerdot$ Commençons par montrer qu'il existe une bijection $T/L\longrightarrow C_2^{24}\times C_3\times \{-1,1\}$. Notons $G=C_2\wr\mathfrak{S}_{24}$ et $H=(C_3\wr\mathfrak{S}_8)\times\mathfrak{S}_{24}$. Alors $T = G\times H$, et en notant $G'=\mathfrak{S}_{24}\subset G$ et $H=\{((\theta,\tau),\phi)\in H|\sum_i\theta_i=0,\varepsilon(\tau)=\varepsilon(\phi)\}\subset H$, on a $L = G'\times H'$.
    
    Or $H/H'\cong C_3\times\{-1,1\}$ car $H'$ est le noyau d'un morphisme surjectif $H \longrightarrow C_3\times \{-1,1\}$. Donc finalement on a une bijection canonique $T/L\cong C_2^{24}\times C_3\times\{-1,1\}$. De plus $I=I_G\times I_H$ où $I_G$ est un sous-groupe de $G$ isomorphe à $I_2^{12}$ et $I_H$ est un sous-groupe de $H$ isomorphe à $\{(0,\Id)\}\times \mathfrak{S}_4^6$. Alors comme $T/L=G/G'\times H/H'$, on a $I\backslash T/L=(I_G\backslash G/G')\times(I_H\backslash H/H')$.
    
    D'après le lemme \ref{IsomI}, $\#I_G\backslash G/G'=3^{12}$. Reste à calculer $\#I_H\backslash H/H'$. On a $I_H\backslash H/H'\cong I_H\backslash(H/H')\cong I_H\backslash(C_3\times\{-1,1\})$. Or l'action de $I_H$ sur $C_3\times\{-1,1\}$ est donnée par $(\phi,(\theta,\epsilon))\longmapsto\phi\cdot(\theta,\epsilon)=(\theta,\varepsilon(\phi)\epsilon)$. C'est une action à trois orbites dont un système de représentants est donné par $(\overline{-1},1),(\overline{0},1) \text{ et } (\overline{1},1)$.
    
    Donc finalement $\#I\backslash T/L = 3^{13}$.
\end{proof}

\begin{Rem}
    On peut déterminer explicitement ces $3^{13}$ représentants, que l'on peut distinguer en fonction des orientations des sommets et des arêtes. En effet il existe essentiellement 3 configurations des sommets :\\
    \begin{figure}[H]
        \centering
        \includegraphics[width = 15cm]{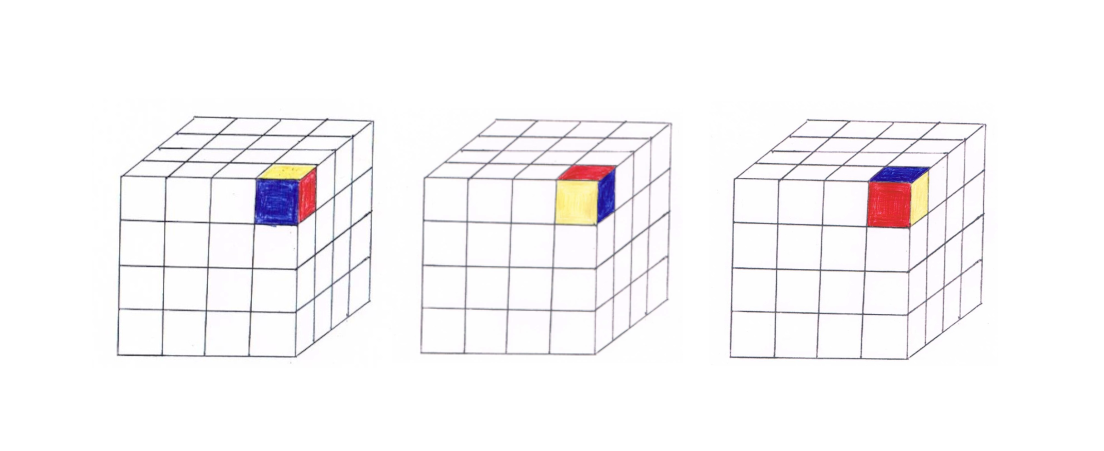}
        \caption{Trois configurations des sommets dans des états inéquivalents différents (aucune modification n'est appliquée sur les autres sommets)}\label{FigConfigSommet}
    \end{figure}
    De plus, il existe essentiellement trois configurations de chaque paire d'arêtes indiscernables :\\
    \begin{figure}[H]
        \centering
        \includegraphics[width = 15cm]{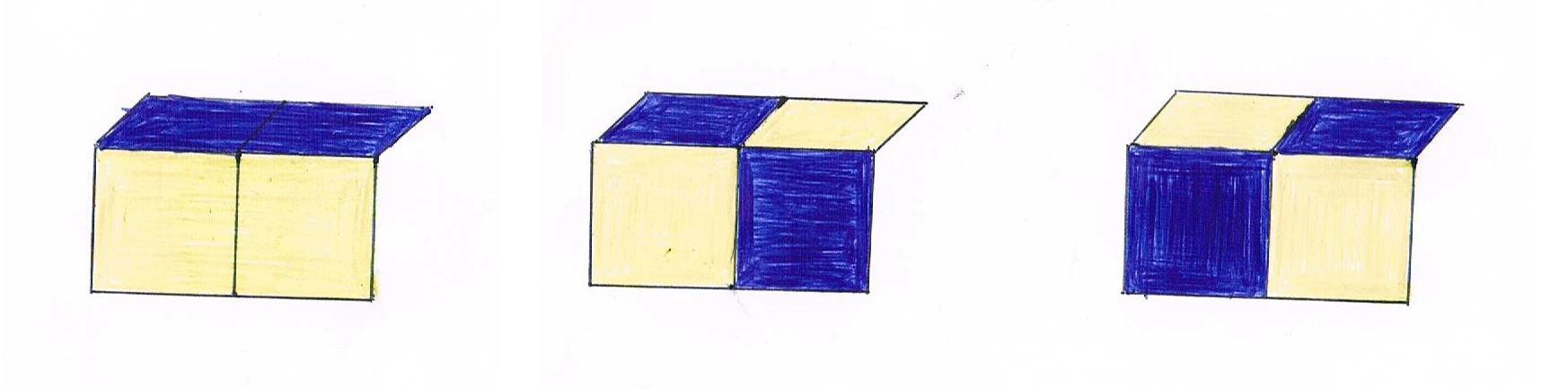}
        \caption{Trois configurations d'arêtes voisines indiscernables dans des états inéquivalents différents (de droite à gauche on trouve les configurations $(\overline{0},\overline{0})$,$(\overline{0},\overline{1})$ et $(\overline{1},\overline{0})$ dans $C_2^2$)}\label{FigConfigArete}
    \end{figure}
    Au total on retrouve des représentants des $3^{13}$ états inéquivalents du Rubik's Revenge.
\end{Rem}

\subsection{Probabilité de pouvoir résoudre le Rubik's Revenge en l'ayant monté au hasard}

Munissons $\Rev_{marq}$ de la probabilité uniforme. La question \textbf{(B)} revient alors à calculer la probabilité de l'événement $g\in \mathcal{I}\cdot x_0\cdot\mathcal{L}$. Par définition de la probabilité uniforme, cette probabilité vaut $\frac{\#\mathcal{I}\cdot x_0\cdot\mathcal{L}}{\#\Rev_{marq}}$.

\begin{Lem}\label{LemProb1}
    On a une bijection $\mathcal{I}\cdot x_0\cdot\mathcal{L}\cong IL$.
\end{Lem}

\begin{proof}
    $x_0$ est envoyé sur le neutre de $T$ par l'isomorphisme de bitorseurs entre $\Rev_{marq}$ et $T$.
\end{proof}

\begin{Lem}\label{LemProb2}
    On a $I\cap L = \{((0,\Id),(0,\Id),\phi)|\phi\in\mathfrak{S}_4^6\cap\mathfrak{A}_{24}\}$.
\end{Lem}

\begin{proof}
    D'une part si $\phi\in\mathfrak{S}_4^6\cap\mathfrak{A}_{24}$, alors $((0,\Id),(0,\Id),\phi)\in I\cap L$.
    
    D'autre part si $((\rho,\sigma),(\theta,\tau),\phi)\in I\cap L$, alors on a $(\theta,\tau)=(0,\Id)$  et donc $\phi\in\mathfrak{S}_4^6\cap\mathfrak{A}_{24}$. Puis $\rho=0$ et donc par définition de $I$, $\sigma=\Id$. D'où l'autre inclusion et l'égalité.
\end{proof}

\begin{Theo}
    La probabilité de pouvoir résoudre un Rubik's Revenge monté au hasard est $\frac{1}{3.2^{12}}$.
\end{Theo}

\begin{proof}
    D'après le lemme \ref{LemProb1}, cette probabilité est donnée par $\frac{\#IL}{\#T}$. Or on a $\#IL=\frac{\#I\#L}{\#I\cap L}$, et d'après le lemme \ref{LemProb2}, $\# I\cap L = \frac{24^6}{2}$. Finalement, comme $L$ est un sous-groupe de $T$ d'indice $6.2^{24}$, on a :
    \begin{align*}
        \frac{\#IL}{\#T} = \frac{\#I\#L}{\#I\cap L.6.2^{24}\#L} = \frac{2^{12}.24^6.2}{24^6.6.2^{24}} = \frac{1}{3.2^{12}}
    \end{align*}
\end{proof}

\section{Prise en compte des contraintes mécaniques}\label{VI}

\subsection{Formalisation des contraintes mécaniques}

Les calculs réalisés dans la section précédente peuvent ne pas paraître entièrement satisfaisants. En effet ces calculs ne prennent pas en compte la mécanique interne du casse tête, en particulier la manière dont il est possible de remonter les pièces des arêtes.

Contrairement aux pièces des centres qui sont toutes identiques, les pièces d'arêtes ne peuvent pas être remontées dans n'importe quel sens après avoir été démontées. Démonter puis remonter le casse-tête préserve en réalité le marquage que l'ont peut observer en figure \ref{FigMarquageAretes}.

Pour tenir compte de ces contraintes mécaniques, on se restreint à l'action à droite sur $\Rev_{marq}$ du sous-groupe de $\mathcal{T}_{geom}$ de toutes les transformations préservant ce marquage, i.e. au sous-groupe $\mathcal{T}'_{geom}$. D'après la proposition \ref{SousBitor}, on définit un nouveau bitorseur :

\begin{Def}
    On appelle ensemble des \emph{états mécaniquement admissibles}, et on note $\Rev_{meca}$ l'ensemble $\Rev_{meca} = x_0\cdot\mathcal{T}'_{geom} \subset \Rev_{marq}$.
\end{Def}

$\Rev_{meca}$ est alors un $(\Ad'_{x_0}(\mathcal{T}'_{geom}),\mathcal{T}'_{geom})$-bitorseur sur lequel viennent naturellement agir deux groupes : à droite les transformations licites $\mathcal{L}$, et à gauche les \emph{transformations indiscernables mécaniquement admissibles} $\mathcal{I}\cap\Ad'_{x_0}(\mathcal{T}'_{geom})$.

\begin{Def}
    Dans la suite on notera $\mathcal{I}'=\mathcal{I}\cap\Ad'_{x_0}(\mathcal{T}'_{geom})$.
\end{Def}

La question \textbf{(A')} se reformule ainsi : trouver le cardinal de $\mathcal{I}'\backslash\Rev_{meca}/\mathcal{L}$. La question \textbf{(B')} revient à calculer $\frac{\#\mathcal{I}'\cdot x_0\cdot\mathcal{L}}{\#\Rev_{meca}}$.

\subsection{Calcul du nombre d'états mécaniquement inéquivalents}

\begin{Def}
    On appelle ensemble des \emph{états mécaniquement inéquivalents} le double quotient $\Ineq_{meca}=\mathcal{I}'\backslash\Rev_{meca}/\mathcal{L}$.
\end{Def}

\begin{Lem}
    On a  une bijection $\Ineq_{meca} \cong I\cap T'\backslash T'/L$.
\end{Lem}

\begin{proof}
    Évidente.
\end{proof}

\begin{Theo}
    On a $\#\Ineq_{meca} = 3$.
\end{Theo}

\begin{proof}
    On a $\#\Ineq_{meca} = \#I\cap T'\backslash T'/L = \#I\cap T'\backslash (C_3\times\{-1,1\})$, et de même que dans la démontration du théorème \ref{QuA}, l'action de $I\cap T'$ possède trois orbites. 
\end{proof}

\begin{Rem}
    Un système de représentants de ces trois états mécaniquement inéquivalents est donné par les configurations des sommets visibles en figure \ref{FigConfigSommet}.
\end{Rem}

\subsection{Calcul de la probabilité en prenant en compte les contraintes mécaniques}

\begin{Lem}
    On a une bijection $\mathcal{I}\cdot x_0\cdot\mathcal{L} \cong (I\cap T')L$.
\end{Lem}

\begin{proof}
    Évidente.
\end{proof}

\begin{Theo}
    Tenant compte des contraintes mécaniques du Rubik's Revenge, la probabilité de pouvoir résoudre le casse-tête monté au hasard est $\frac{1}{3}$.
\end{Theo}

\begin{proof}
    On a :
    \begin{align*}
        \frac{\#\mathcal{I}'\cdot x_0\cdot\mathcal{L}}{\#\Rev_{meca}} = \frac{\#(I\cap T')L}{\#T'} = \frac{\#(I\cap T')\#L}{\#I\cap T'\cap L\#T'} = \frac{\#(I\cap T')\#L}{\#I\cap L\#T'} = \frac{2.24^6\#L}{24^6.6.\#L} = \frac{1}{3}
    \end{align*}
\end{proof}

\begin{Rem}
    En prenant en compte les contraintes mécaniques du casse-tête, on pouvait s'attendre à ce que la probabilité de pouvoir le résoudre en l'ayant monté au hasard soit $\frac{1}{\#\Ineq_{meca}}$.
    
    En effet on n'avait pas de résultat analogue dans le cas sans contrainte car toutes les classes du double quotient $\mathcal{I}\backslash\Rev_{marq}/\mathcal{L}$ n'avaient pas le même cardinal ; en réalité le cardinal d'une telle double classe est uniquement déterminé par la configuration des arêtes d'un état de cette classe. Si on prend en compte les contraintes mécaniques, toutes les doubles classes de $\mathcal{I}'\backslash\Rev_{meca}/\mathcal{L}$ ont un état ayant la même configuration des arêtes que l'état initial $x_0$. Par conséquent toutes ces doubles classes ont le même cardinal, et donc :
    \begin{align*}
        \#\Rev_{meca} = \#(\mathcal{I}'\backslash\Rev_{meca}/\mathcal{
        L}).\#(\mathcal{I}'\cdot x_0\cdot\mathcal{L}) = 3\#\mathcal{I}'\cdot x_0\cdot\mathcal{L},
    \end{align*}
    d'où $\frac{\#\mathcal{I}'\cdot x_0\cdot\mathcal{L}}{\#\Rev_{meca}} = \frac{\#\mathcal{I}'\cdot x_0\cdot\mathcal{L}}{3\#\mathcal{I}'\cdot x_0\cdot\mathcal{L}}=\frac{1}{3}$.
\end{Rem}

\section*{Remerciements}

Je tiens à remercier Benjamin Enriquez pour avoir encadré le mémoire à l'origine de ce papier et pour sa disponibilité, son aide et ses conseils lors de la rédaction de ce travail.


\begin{thebibliography}{5}
    \bibitem[1]{Colmez}
    P. COLMEZ, \textit{Le Rubik's Cube, Groupe de Poche}, 2010
    \bibitem[2]{Giraud}
    J. GIRAUD, \textit{Cohomologie non abélienne, Die Grundlehren der mathematischen Wissenschaften}, Band 179. Springer-Verlag, Berlin-New York, 1971.
    \bibitem[3]{La}
    M. E. LARSEN, \textit{Rubik's Revenge, The Groupe Theoritical Solution}, The American Mathematical Monthly , Jun. - Jul., 1985, Vol. 92
    \bibitem[4]{BonzioRev}
    S. BONZIO, A. LOI, L. PERUZZI, \textit{The first law of cubology for the Rubik's Revenge}, Math. Slovaca 67 no. 3, 2017
    \bibitem[5]{Bonzio}
    S. BONZIO, A. LOI, L. PERUZZI, \textit{On the $n\times n\times n$ Cube}, Math. Slovaca 68 no. 5, 2018
\end{thebibliography}
\end{document}